\documentclass[11pt]{amsart}
\usepackage{graphicx}
\usepackage[mathcal]{euscript}
\usepackage{float}
\usepackage{amsmath,amsthm,amssymb,amsfonts,amscd,epsfig,latexsym,graphicx,textcomp}

\restylefloat{figure}

\newtheorem{Theorem}{Theorem}
\newtheorem{Corollary}[Theorem]{Corollary}

\newtheorem{theorem}{Theorem}[section]
\newtheorem{lemma}[theorem]{Lemma}
\newtheorem{corollary}[theorem]{Corollary}
\newtheorem{proposition}[theorem]{Proposition}

\theoremstyle{definition}
\newtheorem{definition}[theorem]{Definition} 
\newtheorem{example}[theorem]{Example}

\theoremstyle{remark}
\newtheorem{remark}[theorem]{Remark}

\numberwithin{equation}{section}

\newcommand{\ra}{\rightarrow}

\newcommand{\BR}{\mathbb{R}}

\newcommand{\BZ}{\mathbb{Z}}

\newcommand{\ZZ}{{\mathbb Z}}
\newcommand{\RR}{{\mathbb R}}
\newcommand{\CC}{{\mathbb C}}

\begin{document}

\title{On the rotation class of knotted Legendrian tori in $\mathbb{R}^5$}

\author[S. Baldridge]{Scott Baldridge}
\author[B. McCarty]{Ben McCarty}

\thanks{S. Baldridge was partially supported by NSF Grant DMS-0748636.}

\address{Department of Mathematics, Louisiana State University \newline
\hspace*{.375in} Baton Rouge, LA 70817, USA} \email{\rm{sbaldrid@math.lsu.edu}}

\address{Department of Mathematical Sciences, University of Memphis \newline
\hspace*{.375in} Memphis, TN, 38152, USA} \email{\rm{bmmccrt1@memphis.edu}}

\date{}

\begin{abstract}
In this paper we show how to combinatorically compute the rotation class of a large family of embedded Legendrian tori in $\BR^5$ with the standard contact form.  In particular, we give a formula to compute the Maslov index for any loop on the torus and compute the Maslov number of the Legendrian torus. These formulas are a necessary component in computing contact homology.  Our methods use a new way to represent knotted Legendrian tori called Lagrangian hypercube diagrams.
\end{abstract}

\maketitle

\bigskip
\section{Introduction}
\bigskip

Compared to Legendrian knots in $\RR^3$, little is known about knotted Legendrian submanifolds $L^n$ embedded in $\BR^{2n+1}$.  One reason is that in higher dimensions there are no standard representations of embedded Legendrian submanifolds  that enable one to study with the same facility as front projections or Lagrangian projections of Legendrian knots in $\RR^3$.  For example, one may easily compute the classical invariants of Thurston-Bennequin and rotation numbers by looking at the front projection of a knot in $\BR^3$.  Moreover, the classical invariants are quite effective at distinguishing many knots up to Legendrian isotopy:  torus knots, for example have been shown to be classified by their classical invariants (cf. \cite{etnyrehonda}).  

\medskip

While the Thurston-Bennequin number may be generalized to higher dimensions, it is not always as useful as it is for knots in dimension $3$.  In the case we study in this paper, knotted Legendrian tori $L \in \BR^5$, the Thurston-Bennequin invariant is well defined (cf. \cite{T}), but uninteresting since it is always equal to zero.  In fact, the Thurston-Bennequin number in $\RR^{2n+1}$ equals $\frac12\chi(L)$ when $n$ is even.  Furthermore, while topological knot type provides an additional invariant for Legendrian knots in $\RR^3$, all knotted Legendrian surfaces in $\RR^5$ are topologically equivalent provided they are of the same genus.

\medskip

The rotation class is harder to generalize to higher dimensions.  Unlike the Thurston-Bennequin number, which may be defined in terms of a linking number, the rotation number requires the computation of the homotopy class of a map from $L$ to the space of Lagrangians of $\RR^4$ with symplectic structure induced by the contact form on $\RR^5$.  Since writing down this map is non-trivial this invariant is more difficult to compute in higher dimensions.

\medskip

Lagrangian hypercube diagrams overcome the difficulties involved in studying knotted Legendrian tori in $\RR^5$, by providing a way to construct explicit embeddings of Legendrian tori. Using the explicit map defined by a Lagrangian hypercube diagram we demonstrate that the rotation class may be calculated combinatorially as follows:

\bigskip

\begin{Theorem} 
\label{thm:main1}
Given a Lagrangian hypercube diagram $H\Gamma = (C, \{\mathcal{W}, \mathcal{X},\mathcal{Y},\mathcal{Z}\}, G_{zx}, G_{wy})$ with Lagrangian grid diagram projections $G_{zx}$ and $G_{wy}$ in $\RR^2$, and let $L \subset \RR^5$ be the embedded Legendrian torus determined by the lift of the Lagrangian torus defined by $H\Gamma$.  Let $H_1(L) = \langle \tilde{\gamma}_{zx}, \tilde{\gamma}_{wy}\rangle$ be generated by $\tilde{\gamma}_{zx}$ and $\tilde{\gamma}_{wy}$ as in Theorem~\ref{thm:lift}.   Then, the rotation class of $L$, $r(L)$, satisfies:
$$r(L) = (w(G_{zx}), w(G_{wy})),$$
where $w(G_{zx})$ is the winding number of the immersed curve determined by $G_{zx}$.  In particular, the winding number can be computed combinatorically from the Lagrangian grid diagram projection:
$$w(G) = \frac{1}{4} (\#(\text{counterclockwise oriented corners of G}) - \#(\text{clockwise oriented corners of G})).$$
\end{Theorem}

\bigskip

\begin{example}
 
Let $H\Gamma$ be the Lagrangian hypercube diagram constructed from the Lagrangian grid diagrams shown in Figure \ref{fig:5x5unknots} (Theorem~\ref{thm:construction}).  The Lagrangian hypercube determines an immersed Lagrangian torus $T$ (Theorem~\ref{thm:torus}).  The lift of the Lagrangian torus $T$ is a knotted, embedded Legendrian torus $L$ (Theorem~\ref{thm:lift}).   By Theorem \ref{thm:main1}, the rotation class of the Legendrian torus $L$ is $r(L) = (1,0)$.

\begin{figure}[h]
\includegraphics[scale = 1]{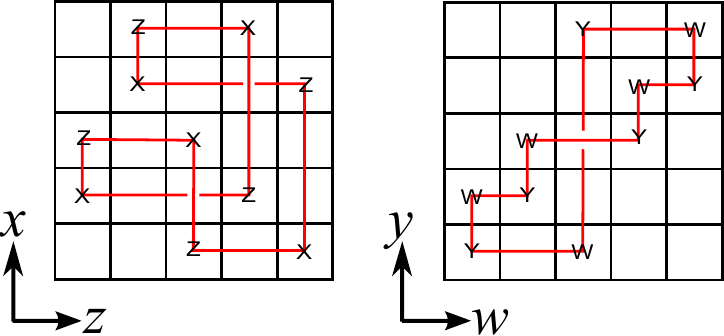}
\caption{Unknots with rotation number $1$ and $0$ respectively..}
\label{fig:5x5unknots}
\end{figure}

\end{example}

Recall that the Maslov index, as defined in \cite{RobSal} and \cite{rotation}, may be viewed as a map $\mu: H_1(L) \rightarrow \ZZ$.  

\begin{Corollary}
\label{cor:main11}
For $(a,b) \in H_1(L)= \langle \tilde{\gamma}_{zx}, \tilde{\gamma}_{wy}\rangle$, the Maslov index is 
$$\mu(A) = 2aw(G_{zx}) + 2bw(G_{wy}).$$
\end{Corollary}

The Maslov number of the torus $L$ is the smallest positive number that is the Maslov index of some nontrivial loop (cf. \cite{rotation}).  Thus Corollary \ref{cor:main11} enables us to compute the Maslov number of $L$ as follows:

\begin{Corollary}
\label{cor:main12}
The Maslov number of $L$ is the non-negative number $2gcd(w(G_{zx}), w(G_{wy}))$.
\end{Corollary}

\medskip

In \cite{rotation},  Ekholm, Etnyre, and Sullivan compute the classical invariants for Legendrian tori obtained by front-spinning, showing that, in particular, the rotation class of the surface so obtained, is determined by the rotation number of the front projection used in the construction.  Thus, their construction leads to tori with rotation class of the form $(0,r)$.  Not only are we able to construct Legendrian tori in which both factors of the torus are knotted, but we show that Legendrian tori constructed from hypercube diagrams realize every possible pair of integers under the isomorphism defined by $H\Gamma$.  In particular, we get examples where the rotation class is $(0,r)$ in the following theorem by taking one of the knots to be a trivial knot with rotation number zero:

\medskip

\begin{Theorem} 
\label{thm:main2}
Let $(m,k) \in \ZZ^2$, and $K_1$, $K_2$ be any two topological knots in $\RR^3$.  Then there is a hypercube diagram, $H\Gamma = (C, \{\mathcal{W}, \mathcal{X},\mathcal{Y},\mathcal{Z}\}, G_{zx}, G_{wy})$ such that $G_{zx}$ and $G_{wy}$ are Lagrangian grid diagrams representing Legendrian knots in $\RR^3$ with the same topological knot type as $K_1$ and $K_2$.  The Legendrian torus $L$ determined by the lift of the Lagrangian torus determined by  $H\Gamma$ satisfies $r(L) = (m,k)$.  
\end{Theorem}

\medskip

Theorem \ref{thm:main2} is a statement about the existence of Lagrangian hypercube diagrams. The methods used in the proof to find Lagrangian hypercube diagrams lead in general to excessively large diagrams.  In practice, however, Lagrangian hypercube diagrams are easy to build by hand.  Knot theory benefited greatly because of the development of nice representations for the knots: braids, knot projections, grid diagrams, etc.  Theorem \ref{thm:main1} and \ref{thm:main2} together can be viewed as our attempt to create similar useful representations of Legendrian tori in $\BR^5$.   In fact, computers can be used to easily generate and compute examples (see Theorem \ref{thm:construction}).

\medskip

This paper stands alone as one of the first papers to explicitly compute classical Legendrian invariants for a large class of knotted Legendrian submanifolds in $\BR^{2n+1}$ for $n\geq 2$ (cf. \cite{rotation}).  We see the potential for much more:  this paper contains key elements in the computing the gradings and dimensions of the moduli spaces used in computing the differential in contact homology.   Our future work will be on how to use the  representations and the calculations in this paper  to  compute the contact homology algorithmically directly from Lagrangian hypercube diagrams.  

\medskip
In fact, we were particularly interested in studying the contact homology of embedded Legendrian tori  in $\RR^5$ (or $S^5$) because of their relationship to Special Lagrangian Cones used to study the String Theory Model in physics.  Briefly, according to this model, our universe is a product of the standard Minkowsky space $\RR^4$ with a Calabi-Yau $3$-fold $X$.  Based upon physical grounds, the SYZ-conjecture of Strominger, Yau, and Zaslov (cf. \cite{SYZ}) expects that this Calabi-Yau 3-fold can be given a fibration by Special Lagrangian $3$-tori with possibly some singular fibers.  To make this idea rigorous one needs control over the singularities, which are not understood well.  One method used to study these singularities (cf. Haskins \cite{Haskins} and Joyce \cite{Joyce}) is to model them locally as special Lagrangian cones $C \subset \CC^3$.  A special Lagrangian cone can be characterized by its associated link $L = C \bigcap S^5$ (the link of the singularity), which turns out to be a minimal Legendrian surface.  When the link type of $L$ is a sphere, then $C$ must be a special Lagrangian plane.  The interesting tractable case appears to be when the link type is an embedded torus.  Several authors (cf. Castro-Urbano \cite{CastroUrbano}, Haskins \cite{Haskins}, Joyce \cite{Joyce}) have shown that there exist infinite families of nontrivial special Lagrangian cones arising from minimal embedded Legendrian tori.  Some work is already being done by Aganagic, Ekholm, Ng, and Vafa \cite{NgSYZ} to understand the connection between contact homology and Lagrangian fillings.  We see this paper as possibly laying groundwork for developing combinatorial tools to understand special Lagrangian cones through the lens of contact homology.    

\medskip

In Section \ref{section:rotation} we present a definition for the rotation class in dimension $5$ and prove that it is characterized by a pair of integers. Section \ref{section:grids} discusses Lagrangian grid diagrams, which enable us to define a Lagrangian hypercube diagram in Section \ref{section:hypercubediagrams}.  In Section \ref{section:torus} we prove that a Lagrangian hypercube diagram represents an immersed Lagrangian torus in dimension 4.  This torus is shown in Section \ref{section:lift} to lift to a Legendrian torus in $\RR^5$ with the standard contact structure.  We then prove Theorem \ref{thm:main1} (Section \ref{section:compute}) and close with a proof of Theorem \ref{thm:main2} and further examples (Section \ref{section:examples}).

\bigskip
\section{Rotation class for embedded Legendrian tori in $\BR^5$}
\label{section:rotation}
\bigskip

In \cite{rotation} the classical Legendrian invariants of Thurston-Bennequin number and rotation number are generalized for $\RR^{2n+1}$.  We recall the definition of rotation class for $\RR^5$ here.  Let $\mathbb{R}^5$ be parametrized using $wxyzt$-coordinates.  Then $\alpha = dt - ydw - xdz$ is a contact $1$-form representing the standard contact structure on $\mathbb{R}^5$.  The contact hyperplanes are given by: 
$$\xi = ker(\alpha) = \{\partial_x, \partial y, \partial_w + y \partial_t, \partial_z + x \partial_t \}.$$

\medskip

Let $f: L \rightarrow (\RR^5, \xi)$ be a Legendrian immersion.  Then the image of $df_x: T_x L \rightarrow T_{f(x)}\RR^5$ is a Lagrangian subspace of the contact hyperplane $\xi_{f(x)}$.  Choose the complex structure $J: \xi_{(w,x,y,z,t)} \rightarrow \xi_{(w,x,y,z,t)}$ such that $J(\partial_w + y \partial_t) = \partial_y$, $J(\partial_y) = -(\partial_w + y \partial_t)$, $J(\partial_z + x \partial_t) = \partial_x$, and $J(\partial_x) = -(\partial_z + x \partial_t)$.  Then the complexification $df_\CC:  TL \otimes \CC \rightarrow \xi$ is a fiberwise bundle isomorphism.  The homotopy class of $(f, df_\CC)$ is called the rotation class of $L$.  Note that the Lagrangian projection $\pi_t:  \RR^5 \rightarrow \CC^4$ gives a complex isomorphism between $(\xi,J)$ and the trivial bundle with fiber $\CC^2$.  Composing $df_\CC$ with $\pi_t$ we get a trivialization $TL \otimes \CC \rightarrow \CC^2$, which we identify with $df_\CC$.  Furthermore, we choose Hermitian metrics on $TL \otimes \CC$ and $\CC^2$ so that $df_\CC$ is unitary.  Thus $f$ gives rise to an element of $U(TL \otimes \CC, \CC^2)$.  The group of continuous maps $C(L,U(2))$ acts freely and transitively on $U(TL \otimes \CC, \CC^2)$ and hence $\pi_0 (U(TL \otimes \CC, \CC^2))$ is in one to one correspondence with $[L,U(2)]$.  From this point forward, we will consider $r(L)$ as an element $[L,U(2)]$.  

\medskip

In general, if $L$ is a genus $g$ Legendrian surface in $\RR^5$, then the rotation class is an element of $[\Sigma_g,U(2)]$.  When $g=0$, $[S^2,U(2)] \cong \pi_2(U(2))$, and hence, the rotation class is always trivial, and uninteresting (for spheres, neither classical invariant yields any useful information).  However, when $g \geq 1$, the rotation class can be nontrivial.  In fact,

\begin{theorem}
\label{isom} The rotation class for a Legendrian torus can be thought of as an element in $\ZZ \times \ZZ$ via the isomorphism $[T,U(2)] \cong \pi_1(U(2)) \times \pi_1(U(2))$.
\end{theorem}

\begin{proof}
Given a map of the standard torus, $i: T^2 \rightarrow \RR^5$, let $a = i(1 \times S^1)$ and $b = i(S^1 \times 1)$.  For $\pi_1(U(2))$, choose basepoint $1 \in U(2)$. Define $H: [T,U(2)] \rightarrow \pi_1(U(2))\times\pi_1(U(2))$ to be the map $f \mapsto (f|_{a},f|_{b})$.  $H$ is surjective since $H(fg)(p,q) = (fg|_a(p),fg|_b(q)) = (f(p),g(q))$ for any pair $f,g \in \pi_1(U(2))$.  The $ker(H)$ is the the set of homotopy classes of maps $f: T \rightarrow U(2)$ such that the $f|_{a \bigcup b}$ is nullhomotopic.  Since $U(2)$ is aspherical, any map such that $f|_{a \bigcup b}$ is nullhomotopic must itself be nullhomotopic.  Hence, the kernel is trivial and $H$ is an isomorphism.
\end{proof}


The existence of the isomorphism in Theorem \ref{isom} is, by itself, not useful in general for calculations due to the fact that the isomorphism depends heavily upon the choice of loops on the torus used to define the map: a generic embedding $i:T^2\ra \BR^5$ does not have a preferred basis for homology (one can precompose with any element of $SL(2,\BZ)$ for example).  However,  Lagrangian hypercube diagrams do provide natural, albeit not canonical, choices for these loops as the torus is embedded in $\BR^5$ (cf. $\tilde{\gamma}_{zx}$ and $\tilde{\gamma}_{wy}$ in Theorem~\ref{thm:lift}).  It is these choices together with Theorem~\ref{isom} that allows us to write down our ``preferred'' calculations of rotation class and Maslov index for loops in the embedded Legendrian torus.  The calculations are important to our future work in computing contact homology of knotted Legendrian tori algorithmically.  While all of our calculations  in computing the contact homology from a Lagrangian hypercube diagram will depend upon these choices, the contact homology calculation in the end will not.      

\medskip

Before moving on to the definition of a Lagrangian hypercube diagram, we begin with a discussion of Lagrangian grid diagrams.

\bigskip
\section{Lagrangian Grid Diagrams}
\label{section:grids}
\bigskip

Let $\mathbb{R}^3$ be given $wyt$-coordinates.  Then $\alpha = dt - ydw$ is a contact $1$-form representing the standard contact structure on $\mathbb{R}^3$.  The contact planes are given by:
$$\xi = ker(\alpha) = \{ \partial_y, \partial_w + y\partial_t \}.$$
A Legendrian knot in $(\RR^3,\xi)$ is an embedding $L: S^1 \rightarrow \RR^3$ whose tangent vectors always lie in the contact planes determined by $\xi$.  Let $\theta \mapsto (w(\theta),y(\theta),t(\theta))$ be a parametrization of $L$.  There are two standard projections used to study Legendrian knots, the front projection:
$$\Pi_L := \Pi \circ L : S^1 \rightarrow \RR^2 : \theta \mapsto (w(\theta),t(\theta)),$$
and the Lagrangian projection:
$$\pi_L := \pi \circ L :  S^1 \rightarrow \RR^2: \theta \mapsto (w(\theta),y(\theta)).$$

\medskip

In general, a given knot diagram will not represent the Lagrangian projection of a Legendrian knot.  However, an immersion $\gamma: S^1 \rightarrow \RR^2: \theta \mapsto (w(\theta),y(\theta))$ will correspond to the Lagrangian projection of a Legendrian knot in $(\RR^3,\xi)$  if the following hold:

\begin{equation}
\label{eqn:areaInt}
\int_0^{2\pi} y(\theta) w'(\theta)d\theta = 0
\end{equation}

\begin{equation}
\label{eqn:crossingInt}
\int_{\theta_0}^{\theta_1} y(\theta) w'(\theta)d\theta \neq 0 \text{ whenever } \theta_0 \neq \theta_1 \text{ and } \gamma(\theta_0) = \gamma(\theta_1).
\end{equation}

We now translate \ref{eqn:areaInt} and \ref{eqn:crossingInt} in the context of grid diagrams.  Let $\hat{G}$ be a $wy$-oriented grid diagram (cf. \cite{Bald}).  Grid diagrams have been studied extensively in \cite{Cromwell}, \cite{knotfloer}, \cite{linkfloer}, \cite{NgArc}, \cite{legend}, and consists of an $n\times n$ grid together with a set of markings that, when connected by edges, represent a knot diagram.  Typically one assigns the $y$-parallel segments in $\hat{G}$ to be the over-strands at any crossing.  However, in the following definition we will ignore such crossing conditions, and think of $\hat{G}$ as an immersed $S^1$.  

\begin{definition}
An \emph{immersed grid diagram} is an oriented grid diagram $G$ with no crossing data specified.
\end{definition}

An immersed grid diagram $G$ may be thought of as a mapping $\gamma: S^1 \rightarrow \RR^2: \theta \mapsto (w(\theta),y(\theta))$.  Since $w'(\theta)$ is $0$ along any segment in $G$ parallel to the $y$-axis, and $y(\theta)$ is constant along any segment parallel to the $w$-axis, Condition \ref{eqn:areaInt} translates into 
$$\int_{0}^{2\pi}y(\theta)w'(\theta)d\theta = \sum_{i=1}^n {\sigma(a_i) \cdot y_i \cdot length(a_i)} = 0,$$ 
where $\{a_i\}$ is the collection of segments of $G$ parallel to the $w$-axis, $y_i$ is the $y$-coordinate of $a_i$, and $\sigma(a_i)$ is $+1$ if $a_i$ is oriented left to right and $-1$ otherwise.  Given a crossing in $G$ (i.e. given $\theta_0 < \theta_1$ such that $\gamma(\theta_0) = \gamma(\theta_1)$), Condition \ref{eqn:crossingInt} becomes:
$$\int_{\theta_0}^{\theta_1}y(\theta)w'(\theta)d\theta = \sum_{i=1}^m{\sigma(a_i) \cdot  y_i \cdot length(c_i)} \neq 0,$$
where $\{c_i\}$ is the set of $w$-parallel segments in the loop beginning and ending at the given crossing and such that $\gamma(\theta) \neq \gamma(\theta_0)$ for all $\theta \in (\theta_0,\theta_1)$.  Condition \ref{eqn:areaInt} guarantees that choosing the other loop $(\theta_1, \theta_0) \in \RR/{2\pi\ZZ}$ will give the same integral up to sign as the one chosen.  Therefore any immersed grid diagram $G$ satisfying Conditions (1) and (2) lifts to a piecewise linear Legendrian knot in $(\RR^3,\xi)$ as follows:  choose some $\theta_0 \in S^1$ and define the $t$-coordinate $t_0$ of $\gamma(\theta_0)$ to be $0$.  Then define 

\begin{equation}
\label{eqn:lift}
t_\theta = t_0 + \int_{\theta_0}^{\theta} y(u)w'(u)du.
\end{equation}

Condition \ref{eqn:areaInt} guarantees that in defining the $t$-coordinate this way, the lift will be a closed loop.  Condition \ref{eqn:crossingInt} guarantees that the vertical and horizontal segments at a crossing will have different $t$-coordinates.  

\begin{definition}
A \emph{Lagrangian grid diagram} is an immersed grid diagram $G$ satisfying Conditions \ref{eqn:areaInt} and \ref{eqn:crossingInt}.
\end{definition}

Given a Lagrangian projection of a Legendrian knot $L$, one may compute the rotation number as follows.  Use the vector field $w = \frac{\partial}{\partial y}$ to trivialize $\xi|_L$.  Then the rotation number may be calculated to be the winding number of the tangent vector to $L$ with respect to this trivialization:
$$r(L) = w(\pi_L).$$
For a Lagrangian grid, this is simply a signed count of the corners of $G$.  Let $B$ be the collection of corners in $G$.  Then for a corner $b \in B$ let $\eta(b)$ be a function that assigns a value of $+1$ to any corner of type $W:{NE}$, $Y:{NW}$, $W:{SW}$, and $Y:{SE}$ (i.e. a counterclockwise oriented corner), and a value of $-1$ to any corner of type $W:{NW}$, $Y:{NE}$, $W:{SE}$, and $Y:{SW}$ (i.e. a clockwise oriented corner) following the same notation as in \cite{legend} and \cite{NgThurston}.  Figure \ref{fig:bendtypes} illustrates the types of corners.  Thus we observe that:

\begin{lemma}
\label{lemma:rotation}
Given a Lagrangian grid diagram $G$ with Legendrian lift $L$, the rotation number satisfies:
$$r(L) = w(G) = \frac{1}{4}\sum_{b \in B} \eta(b).$$   
\end{lemma}
 
\begin{figure}[h]
\includegraphics[scale = .4]{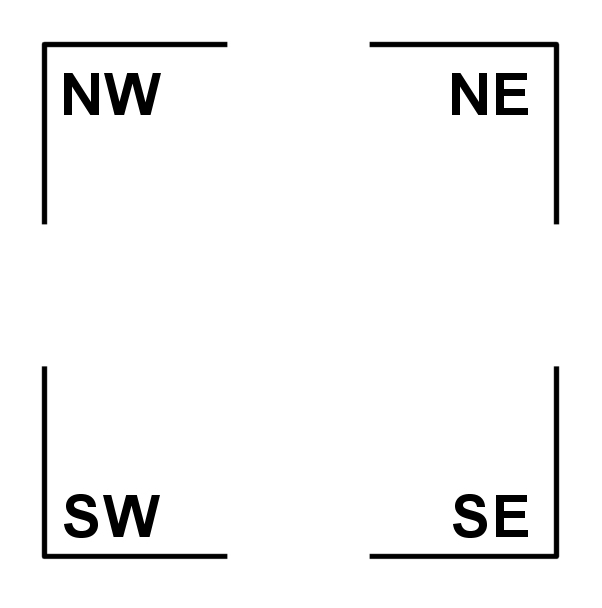}
\caption{Types of corners in a grid diagram.}
\label{fig:bendtypes}
\end{figure}

\begin{figure}[h]
\includegraphics[scale = .8]{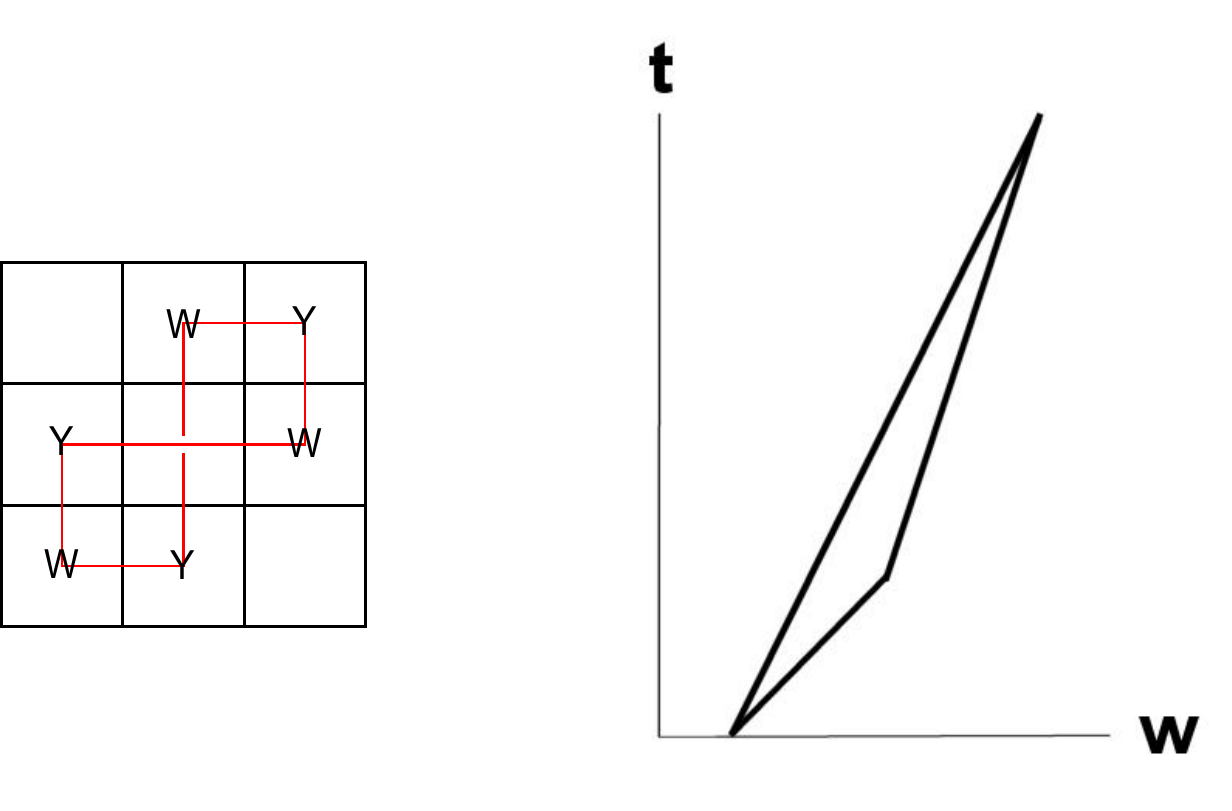}
\caption{A $wy$ immersed grid diagram for the unknot and its corresponding front projection.}
\label{fig:unknot}
\end{figure}

\begin{example}
\label{ex:r0}

Observe that $\int_G ydw = \frac{3}{2} + \frac{7}{2} - 2(\frac{5}{2}) = 0$, and for a path connecting the crossing to itself, $\int_G ydw = \frac{7}{2} - \frac{5}{2} = 1$.  Hence, the unknot shown in Figure \ref{fig:unknot} is a Lagrangian grid.  Set the $t$-coordinate of the $w$-mark in column $1$ to $0$ and define the lift as in Equation \ref{eqn:lift}.  Then the front projection corresponding to the lift of $G$ is shown in Figure \ref{fig:unknot}.  The rotation number is easily computed from this projection since $G$ has $3$ bends that are assigned a value of $+1$ and $3$ that are assigned a value of $-1$.  Hence, $r(G) = 0$.

\end{example}

The Legendrian knots produced using the above method will be piecewise linear, not smooth.  However, we can produce smoothly embedded knots as follows.  Choose $0 < \epsilon << 1$.  Delete an $\epsilon$ neighborhood of each vertex of $G$ and replace it with a smooth curve (cf. Figure \ref{fig:smooth1}).  Such a smoothing may be accomplished so as to guarantee that the diagram is smooth at the boundary of the $\epsilon$ neighborhood as well.  For example, the image of the map
$$E(t) = (w + \epsilon - \epsilon cos(t/\epsilon),y + \epsilon - \epsilon sin(t/\epsilon)).$$
allows one to replace a $W:{SE}$ corner with a smooth arc, but the resulting rounded corner will only be $C^1$ at the boundary of the $\epsilon$ neighborhood.  Note that the smoothing may be done so that the resulting curve is symmetric about the line of slope $\pm1$ through the vertex of the bend. Furthermore, given a choice of a smoothing at a corner such that the area enclosed by the smooth curve and the original bend is $A$, one may obtain a different smoothing so that the area enclosed is $rA$ where $r \in \RR$ such that $0 < r \leq 1$.  

\begin{figure}[h]
\includegraphics[scale = .7]{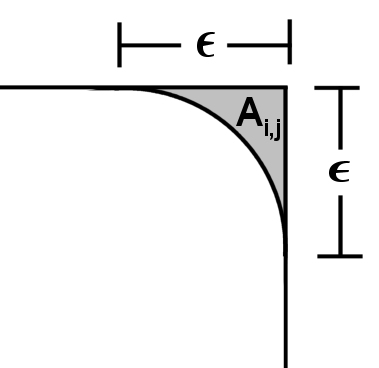}
\caption{A smoothing of a corner.}
\label{fig:smooth1}
\end{figure}

\begin{proposition}
\label{prop:smooth}
Let $\gamma: S^1 \rightarrow \RR^2$ be the piecewise linear immersion determined by the Lagrangian grid diagram, $G$.  There exists a $\delta > 0$ such that for any $0 < \epsilon  \leq \delta$ there is a choice of smoothing curves based upon $\epsilon$ such that the immersion determined by the smoothed grid, $\gamma_\epsilon: S^1 \rightarrow \RR^2$ satisfies the following:
\begin{itemize}
\item the lift of $\gamma_\epsilon$ is $C^0$-close to the lift of $\gamma$, and
\item for any two $\epsilon$, $\epsilon' < \delta$ the Legendrian knots $K$, $K'$ are Legendrian isotopic.
\end{itemize}

\end{proposition}  

\begin{proof}
Choose $\delta > 0$ such that $\delta^2 < \frac{1}{2n}$.  Let $\epsilon < \delta/2$.  Enumerate the corners $b_{i,j} \in B$ so that corner $b_{i,1}$ is the corner on the lefthand side of row $i$ and $b_{i,2}$ is the corner on the righthand side of row $i$.  Let $A_{i,j}$ be the absolute value of the area of the region enclosed by the smoothed arc and the original corner of the corner $b_{i,j} \in B$.  Construct each smoothing so that $|A_{i,j}| \leq \epsilon$.  Denote by $r_i$ the horizontal edge in row $i$.  Then we have the following:
$$\int_{G} ydw = \sum_{i = 1}^{n} {\sigma(r_i) \cdot (i \cdot length(r_i) - \tau_1(i)A_{i,1} - \tau_2(i)A_{i,2})} = -\sum_{i = 1}^{n} {\sigma(r_i) \cdot (\tau_1(i)A_{i,1} + \tau_2(i)A_{i,2})}$$
where $\sigma(r_i)$ is $+1$ if the edge is directed left to right and $-1$ otherwise, $\tau_j(i)$ is $+1$ if the smoothing lies above the horizontal edge, and $-1$ otherwise.

\medskip

Since not all of $\sigma(r_i) \cdot \tau_1(i)$ will evaluate to $+1$ (respectively, all $-1$), we may choose the smoothings so that 
$$\sum_{i = 1}^{n} {\sigma(r_i) \cdot (\tau_1(i)A_{i,1} + \tau_2(i)A_{i,2})} = 0.$$

\medskip

Since the value of the integral in Equation \ref{eqn:crossingInt} may only change from the piecewise linear calculation by an amount less than $\frac{1}{4}$, the smoothed diagram has the same crossing data as the original Lagrangian grid diagram.  The second condition of the Lemma is clear.
\end{proof}

Note if $A_{i,j} = A$ for all $i,j$, the above sum evaluates to $4Ar(G)$.  Thus, if the rotation number is $0$ then the same smoothing may be used for all vertices of $G$.  

\begin{corollary}
\label{cor:smoothing}
Let $\gamma_\epsilon$ be parametrized by $\theta \mapsto (w(\theta),y(\theta))$. Then, 
$$\bigg|\int_{\theta_0}^{\theta_1} y(\theta)w'(\theta)d\theta - \int_{\theta_0}^{\theta_1} y_\epsilon(\theta)w'_\epsilon(\theta)d\theta \bigg| < \frac{1}{4}.$$
\end{corollary}

Proposition \ref{prop:smooth} and Corollary \ref{cor:smoothing} show that a Lagrangian grid diagram corresponds to a smoothly embedded Legendrian knot that does not depend on the choice of epsilon used in the smoothing.  Hence we may refer to \emph{the Legendrian knot} corresponding to a Lagrangian grid diagram.

\begin{example}
Since the rotation number of the unknot in Figure \ref{fig:unknot} is $0$ we may choose to smooth all corners in the same way, thus obtaining a Lagrangian projection of a smoothly embedded Legendrian knot in $(\RR^3,\xi)$.  
\end{example}

\begin{figure}[h]
\includegraphics[scale = .8]{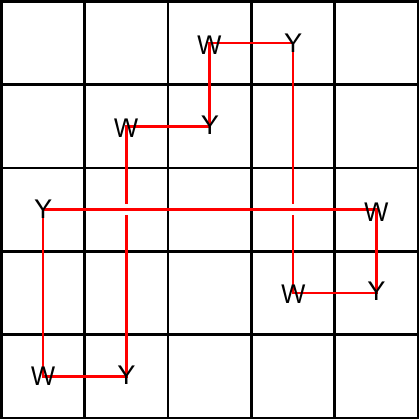}
\caption{An unknot with rotation number $1$.}
\label{fig:unknot2}
\end{figure}

\begin{example}
\label{ex:r1}
The unknot shown in Figure \ref{fig:unknot2} may easily be seen to have rotation number $1$.  In order to smooth the diagram, we perform the following calculation.  To simplify matters choose the smoothings so that the areas satisfy $A_{i,1} = A_{i,2}$, $A_{i} = A_{i,j}$, and all are less than $\frac{1}{100}$.  
$$\sum_{i = 1}^{n} {\sigma(r_i) \cdot (\tau_1(i)A_i + \tau_2(i)A_i)} = 2A_1 + 2A_2 +  2A_3 - A_4 + A_4 - 2A_5$$
Choose the $A_i$ so that $A_1 = A_2 = A_3$ and $A_5 = 3A_1$.  Then this sum will be $0$ and the Lagrangian grid conditions will still be satisfied by the smoothed diagram, and the diagram will be the Lagrangian projection of a smoothly embedded Legendrian knot in $(\RR^3,\xi)$.  
\end{example}

The Legendrian lift of the smoothed Lagrangian grid diagram is unique up to Legendrian isotopy (Proposition \ref{prop:smooth}).  By Corollary \ref{cor:smoothing} we can do integer calculations directly from the Lagrangian grid diagram instead of the smooth $\gamma_\epsilon$ loop, without worrying about changing the crossing information of the lift of the Lagrangian grid diagram.  In particular, there is a correspondence of horizontal edges with opposite orientation in each column that allows one to re-interpret the Lagrangian grid conditions as a signed area sum.  That is:  

\begin{corollary}
\label{cor:areaCalc}
There is a set of rectangles (possibly overlapping) with horizontal edges lying on the knot diagram whose signed areas sum to the same value as the integral in Equation \ref{eqn:lift}.  
\end{corollary}
 
\begin{example}
For the grid diagram in Figure \ref{fig:areaCalc}, we see by computing the signed areas shown that the integral in Equation \ref{eqn:lift} evaluate to $-7$.  Hence, it is not a Lagrangian grid diagram.

\begin{figure}[h]
\includegraphics[scale = 7]{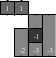}
\caption{Decomposition of grid into rectangles.}
\label{fig:areaCalc}
\end{figure}

\end{example}

In practice, the area calculation described in the previous example may be carried out by simply decomposing the grid into polygonal regions where the top-most horizontal edges are all oriented left (resp. right) and the bottom-most horizontal edges are all oriented right (resp. left).  Then, the signed area of these polygonal regions will correspond to the integrals defined in Conditions \ref{eqn:areaInt} and \ref{eqn:crossingInt}.  For convenience, in the proofs that follow, we will use this signed area calculation to compute the integrals defined in Conditions \ref{eqn:areaInt} and \ref{eqn:crossingInt}.

\begin{theorem}
\label{thm:anyRotation}
Any topological knot type with any rotation number may be realized as a Lagrangian grid diagram.
\end{theorem}

Before proving the theorem, we introduce some definitions and lemmas that we will use only for the proofs in this paper.

\begin{definition}
\label{almostLag}
 An \emph{almost Lagrangian grid diagram} is an immersed grid diagram such that:
\begin{itemize}
 \item the top right corner has a marking,
 \item there is a parametrization $\gamma: I \rightarrow G \subset \RR^2$ in which $\gamma(\theta)$ starts and ends at that marking point.
 \item $\int_{\theta_1}^{\theta^2} y(\theta)w'(\theta)d\theta \neq 0$ whenever $\theta_1,\theta_2 \in (0,1)$, $\theta_1 \neq \theta_2$ and $\gamma(\theta_1) = \gamma(\theta_2)$.
\end{itemize}

\end{definition}

Let the $t$-coordinate of $\gamma(0)$ be $0$.  Then define,
$$t_\theta = \int_{0}^{\theta} y(u)w'(u)du.$$
Thus the last condition of Definition \ref{almostLag} guarantees that an almost Lagrangian grid diagram gives rise to an embedded Legendrian arc. Since the endpoints of this arc project to the top right corner marking and differ only in their $t$-coordinates, an almost Lagrangian grid diagram still gives rise to a knot in $\RR^3$ by attaching the endpoints by a segment parallel to the $t$-axis.  

\begin{lemma}
\label{lemma:capArc}
 An almost Lagrangian grid diagram can always be modified (using configurations listed in Table \ref{corners}) to get a Lagrangian grid diagram with the same topological knot type and winding number as the knot given by the almost Lagrangian grid diagram.  
\end{lemma}

\begin{proof}
An almost Lagrangian grid diagram represents a Legendrian arc whose endpoints have $t$-coordinates that differ by some $k \in \ZZ$.  Attach one of the configurations shown in Table \ref{corners}.  Each time such a configuration is attached, the resulting grid will again be an almost Lagrangian grid diagram, but the difference between the end points of the new Legendrian arc will be reduced by $1$ or $2$. Continue reducing this difference until the arc closes up to give a Lagrangian grid diagram.
\end{proof}

\begin{table}[h]
\begin{tabular}{c}
\includegraphics[scale = 1]{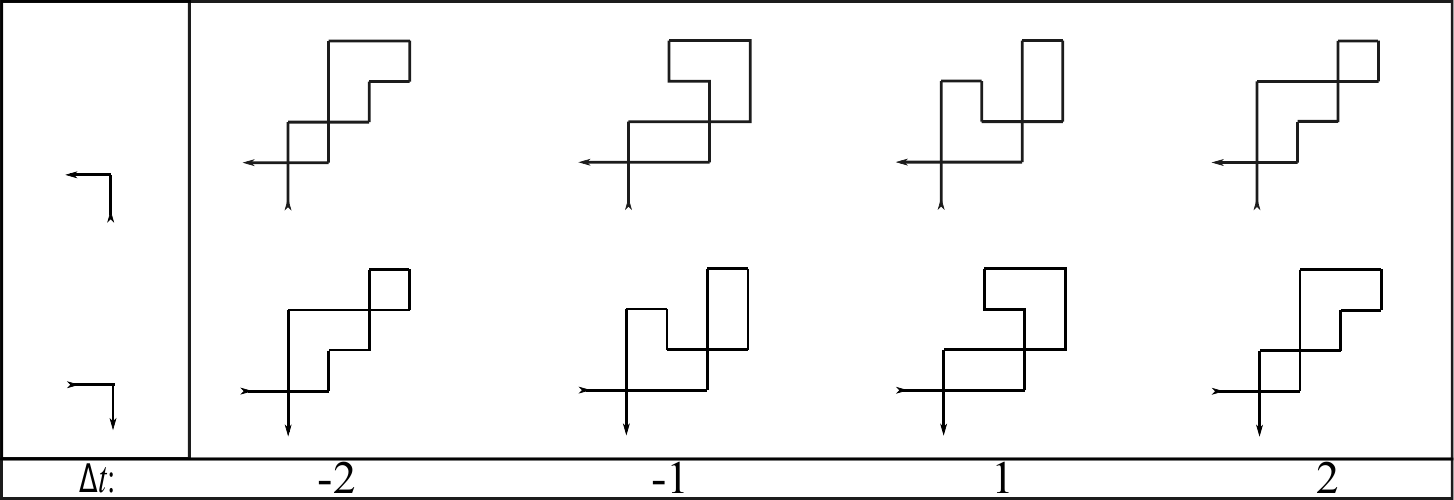}
\end{tabular}
\caption{Configurations used to convert an almost Lagrangian grid diagram into a Lagrangian grid diagram.  The value of $\Delta t$ follows from Corollary \ref{cor:areaCalc}.}
\label{corners}
\end{table}

\begin{lemma}
\label{lemma:anyRotation}
Let $k \in \ZZ$.  Any Lagrangian grid diagram can be modified to obtain a Lagrangian grid diagram with rotation number $k$.
\end{lemma}

\begin{proof}

Let $k \in \ZZ$.  If the Lagrangian grid diagram does not have a marking in the top right corner, modify it so that does by stabilizing in the righthand column and commuting the horizontal edge of length $1$ to the top of the grid, to obtain an almost Lagrangian grid diagram.  Then, at this top right corner, attach one of the configurations shown in Figure \ref{fig:rotationChange} to change the rotation number to $k$.  This new object is an almost Lagrangian grid diagram.  Apply Lemma \ref{lemma:capArc} to obtain a Lagrangian grid diagram whose lift has the same topological knot type as the original Lagrangian grid diagram.
\end{proof}

\begin{figure}[h]
\includegraphics[scale=.7]{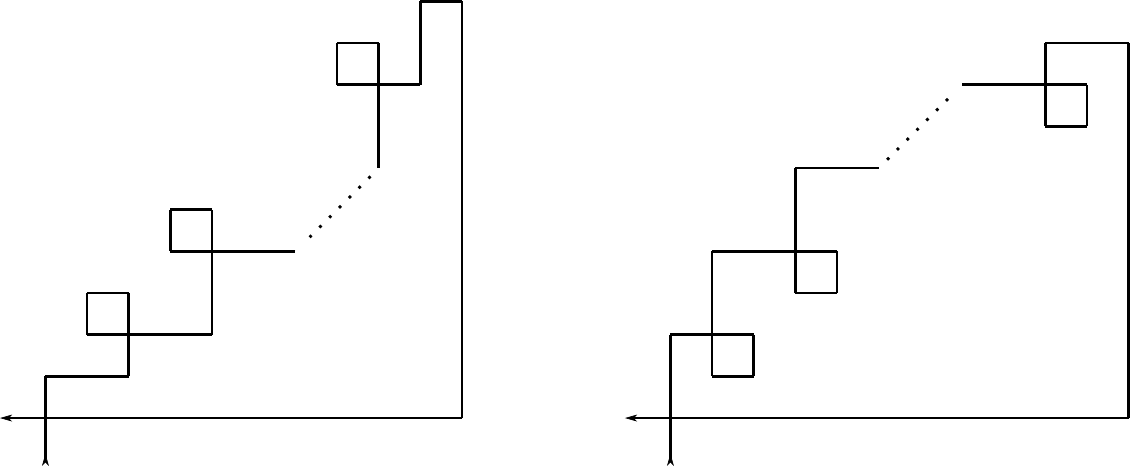}
\caption{Configuration to change the winding number of an immersed grid diagram.}
\label{fig:rotationChange}
\end{figure}

We now proceed with the proof of Theorem \ref{thm:anyRotation}.

\begin{proof}
We use Lenhard Ng's arguments, \cite{Ng}, as a guide to construct Lagrangian grid diagrams.  Recall that a grid diagram (in the usual sense) may be thought of as a front projection of a Legendrian knot.   Given such a front projection, we may resolve the front to obtain the Lagrangian projection of a knot isotopic to the one determined by the front.  This Lagrangian projection will have the same crossing data as the original grid, and, as a diagram, is isotopic to the original grid after adding loops at each southeast corner.  

We follow a similar procedure, but modify it so that we obtain a Lagrangian grid diagram.  Given a grid diagram (in the usual sense), stabilize at each southeast corner (without adding a crossing), and commute the horizontal edge of length $1$ to the bottom of the grid to obtain a simple front (cf. \cite{Ng}).  By applying another stabilization in the right-most column, and then commutation moves, we may ensure that this grid has a marking in the top right corner.  Then add a loop at each southeast corner, as is done in constructing the front resolution.  By possibly inserting some number of empty rows and columns, we may adjust the enclosed areas so that we obtain a diagram whose lift represents the same knot in $\RR^3$ as the grid diagram we started with.  This diagram will, in general, not be a grid diagram, since it contains empty rows and columns.  At the top right corner, attach a configuration as shown in Figure \ref{fig:gridFill} to fill in any empty rows and columns, and thus obtain an almost Lagrangian grid diagram.  Then, by applying Lemmas \ref{lemma:capArc} and \ref{lemma:anyRotation}, we may obtain a Lagrangian grid diagram representing the same topological knot type as the original grid diagram, and having any rotation number $k$.
\end{proof}

\begin{figure}[h]
\includegraphics[scale = 1]{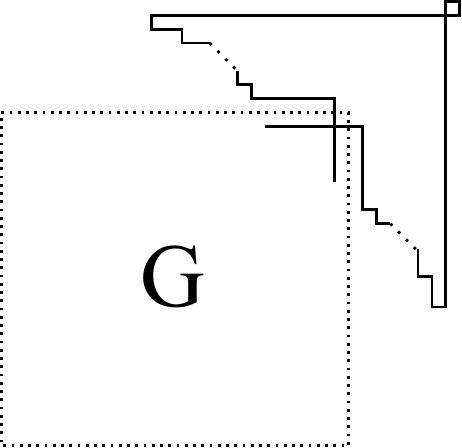}
\caption{Filling in empty rows and columns.}
\label{fig:gridFill}
\end{figure}

\bigskip
\section{Lagrangian hypercube diagrams in dimension 4}
\label{section:hypercubediagrams}
\bigskip

The definition of a Lagrangian hypercube diagram codifies a data structure that mimics  that of hypercube diagrams, cube diagrams and grid diagrams.  While the definition appears similar to that of $4$-dimensional hypercube diagrams as defined in \cite{Bald}, they are not equivalent.  Let $n$ be a positive integer and  let the hypercube $C = [0,n]\times [0,n]\times [0,n] \times [0,n] \subset \mathbb{R}^4$ be thought of as a $4$-dimensional Cartesian grid, i.e., a grid with integer valued vertices with axes $w$, $x$, $y$, and $z$.  Orient $\BR^4$ with the orientation $w\wedge x\wedge y\wedge z$.

\medskip

 A \textit{flat} is any right rectangular $4$-dimensional prism with integer valued vertices in the hypercube such that there are two orthogonal edges at a vertex of length $n$ and the remaining two  orthogonal edges are of length $1$.  Name flats by the axes parallel to the two orthogonal edges of length $n$.  For example, a $yz$-flat is a flat that has a face that is an $n\times n$ square that is parallel to the $yz$-plane.

\medskip

Similarly, a {\em cube} is any right rectangular $4$-dimensional prism with integer vertices in the hypercube such that there are three orthogonal edges of length $n$ at a vertex with the remaining orthogonal edge of  length $1$.  Name cubes by the three  edges of the cube of length $n$.  See Figure~\ref{cubesandflats} for examples.

\medskip

A marking is a labeled point in $\BR^4$ with half-integer coordinates.  Mark unit hypercubes in the $4$-dimensional Cartesian grid with either a $W$, $X$, $Y$, or $Z$ such that the following {\em marking conditions} hold:

\begin{itemize}
 \item each cube has exactly one $W$, one $X$, one $Y$, and one $Z$ marking;\\

 \item each cube has exactly two flats containing exactly 3 markings in each; \\

 \item for each flat containing exactly 3 markings, the markings in that flat form a right angle such that each ray is parallel to a coordinate axis;\\

    \item for each flat containing exactly 3 markings, the marking that is the vertex of the right angle is $W$ if and only if the flat is a $zw$-flat, $X$ if and only if the flat is a $wx$-flat, $Y$ if and only if the flat is a $xy$-flat, and $Z$ if and only if the flat is a $yz$-flat.

\end{itemize}

\begin{figure}[h]
\includegraphics[scale=.5]{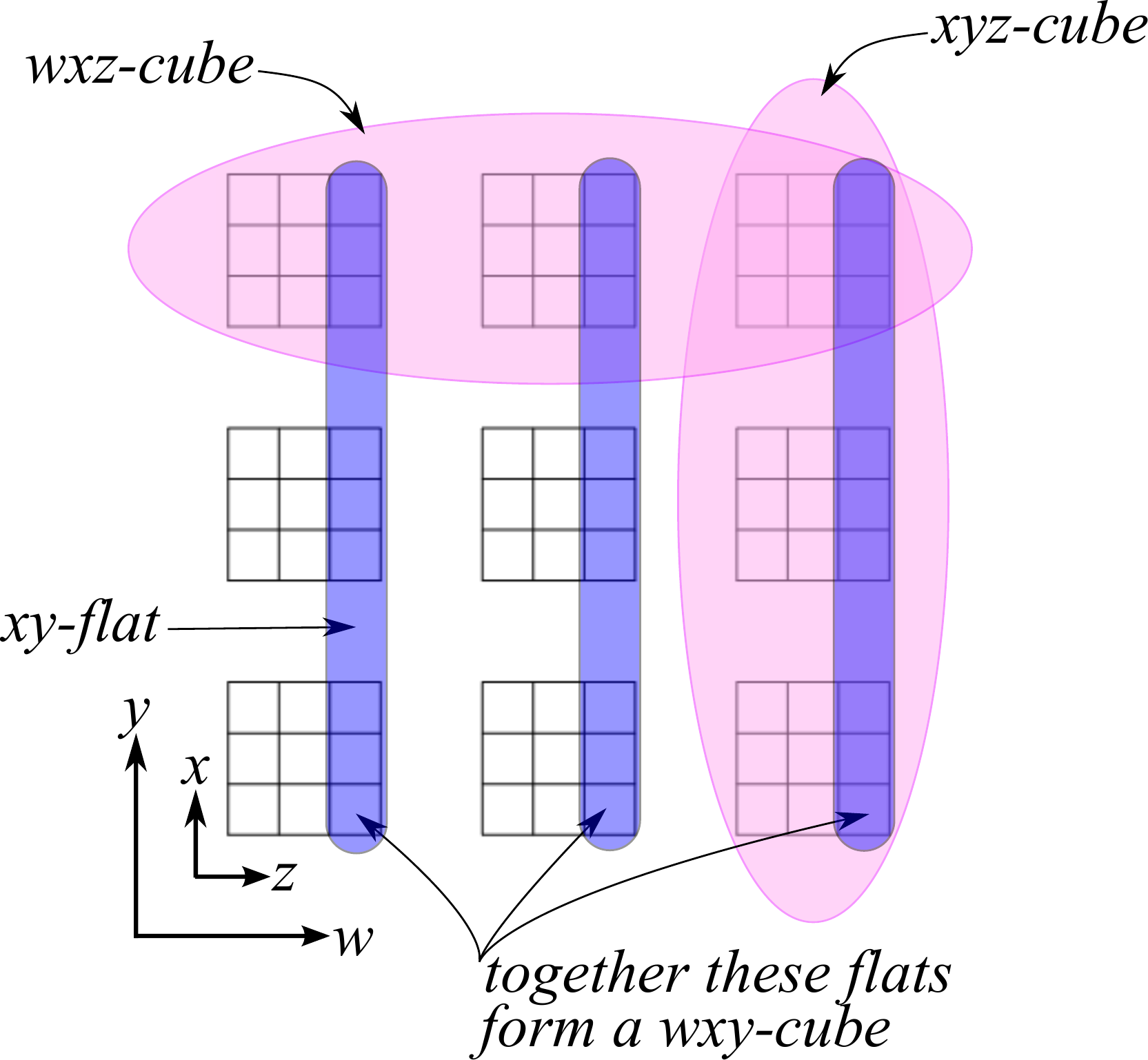}
\caption{\small \it  A schematic for displaying a Lagrangian hypercube diagram.  The outer $w$ and $y$ coordinates indicate the ``level'' of each $zx$-flat.  The inner $z$ and $x$ coordinates start at $(0,0)$ for each of the nine $yz$-flats. With these conventions understood, it is then easy to display $xy$-flats, $xyz$-cubes, $wxz$-cubes, $wxy$-cubes, etc. } \label{cubesandflats}
\end{figure}

\medskip

The 4th condition rules out the possibility of either $wy$-flats or a $zx$-flats with three markings.  As with oriented grid diagrams and cube diagrams, we obtain an oriented link from the markings by connecting each $W$ marking to an $X$ marking by a segment parallel to the $w$-axis, each $X$ marking to a $W$ marking by a segment parallel to the $x$-axis, and so on.
   
\medskip

Let $\pi_{xz}, \pi_{wy} : \RR^4 \rightarrow \RR^2$ be the natural projections.  Define $G_{wy} := \pi_{xz}(C)$ and $G_{zx} := \pi_{wy}(C)$ which are immersed grid diagrams.  Let $\{c_i\}$ be the crossings in $G_{zx}$, and $\{c_i'\}$ be the crossings in $G_{wy}$.  Then we say that the \emph{Lagrangian crossing conditions} hold for the pair $G_{zx}$ and $G_{wy}$ if $|\Delta t(c_i)| \neq |\Delta t(c_i')|$ $\forall i,j$ where $\Delta t$ is the difference in the $t$-coordinates at each crossing determined by Equation \ref{eqn:crossingInt}.  

\medskip

\begin{definition}
If the markings $\{\mathcal{W}, \mathcal{X},\mathcal{Y},\mathcal{Z}\}$ in $C$ satisfy the marking conditions, and the immersed grid diagrams $G_{wy}$ and $G_{zx}$ are Lagrangian grid diagrams satisfying the Lagrangian crossing conditions, then we define $H\Gamma = (C, \{\mathcal{W}, \mathcal{X},\mathcal{Y},\mathcal{Z}\}, G_{zx}, G_{wy})$ to be a \emph{Lagrangian hypercube diagram}.  
\end{definition}

\bigskip
\section{Building a torus from a Lagrangian hypercube diagram}
\label{section:torus}
\bigskip

A hypercube schematic (cf. Figure \ref{fig:unknot3}) conveniently displays the markings of a Lagrangian hypercube diagram so that the Lagrangian grid diagrams $G_{zx}$ and $G_{wy}$ may be read off of the diagrams directly.  To see $G_{wy}$ treat each $n \times n$ $zx$-flat as a cell of $G_{wy}$ (i.e. consider the projection $\pi_x \circ \pi_z$).  Each $zx$-flat containing a $W$ and $Z$ marking will project to a cell of $G_{wy}$ containing a $W$ marking and each $zx$-flat containing an $X$ and $Y$ marking will project to a cell of $G_{wy}$ containing a $Y$ marking.  In Figure \ref{fig:unknot3}, the blue shading indicates the diagram associated to $G_{wy}$.  To see $G_{zx}$ in the schematic, note that each pair of markings in a $zx$-flat on the schematic corresponds to an edge of the Lagrangian grid diagram $G_{zx}$.  Placing these segments on a single $n \times n$ grid will produce a copy of $G_{zx}$.

\medskip

To produce an immersed torus from the Lagrangian hypercube diagram, place a copy of the immersed grid $G_{zx}$ at each $zx$-flat on the schematic that contains a pair of markings (shown in red on Figure \ref{fig:unknot3}).  Doing so produces a schematic with two copies of $G_{zx}$ with the same $y$-coordinates and two with the same $w$-coordinates.  For each pair of copies sharing the same $w$-coordiantes, we may translate one parallel to the $w$-axis toward the other.  Doing so traces out an immersed tube connecting these two copies of $G_{zx}$.  Similarly, we may translate parallel to the $y$-axis to produce an immersed tube connecting two copies of $G_{zx}$ with the same $y$-coordinates.  Since we are connecting copies of $G_{zx}$ in flats corresponding to the markings of $G_{wy}$, the tube will close to produce an immersed torus.  Thus we obtain:

\begin{theorem}
\label{thm:torus}
A Lagrangian hypercube diagram determines an immersed Lagrangian torus $i: T \rightarrow \RR^4$.  Furthermore, the map determines a preferred set of loops, $\gamma_{zx} = S^1 \times 1$ and $\gamma_{wy} = 1 \times S^1$, that map to curves projecting to the Lagrangian grid diagrams $G_{zx}$ and $G_{wy}$.
\end{theorem}

Since the torus is formed by the translation of $x$ and $z$-parallel segments to the $w$ and $y$ axes, we see that only $wx$, $wz$, $yz$, and $xy$ rectangles are used in the construction of the torus.  Since $wy$ and $zx$ rectangles are never used in the construction of the torus, it is Lagrangian with respect to the symplectic form $dw \wedge dy + dz \wedge dx$.  Furthermore, just as in the case of Lagrangian grid diagrams, we obtained a smooth embedding by carefully smoothing corners, we may obtain a smooth embedding of the torus in $\RR^5$ by first smoothing $G_{zx}$ and $G_{wy}$ as in Lemma \ref{prop:smooth}.

\medskip

\begin{figure}[h]
\includegraphics[scale=.6]{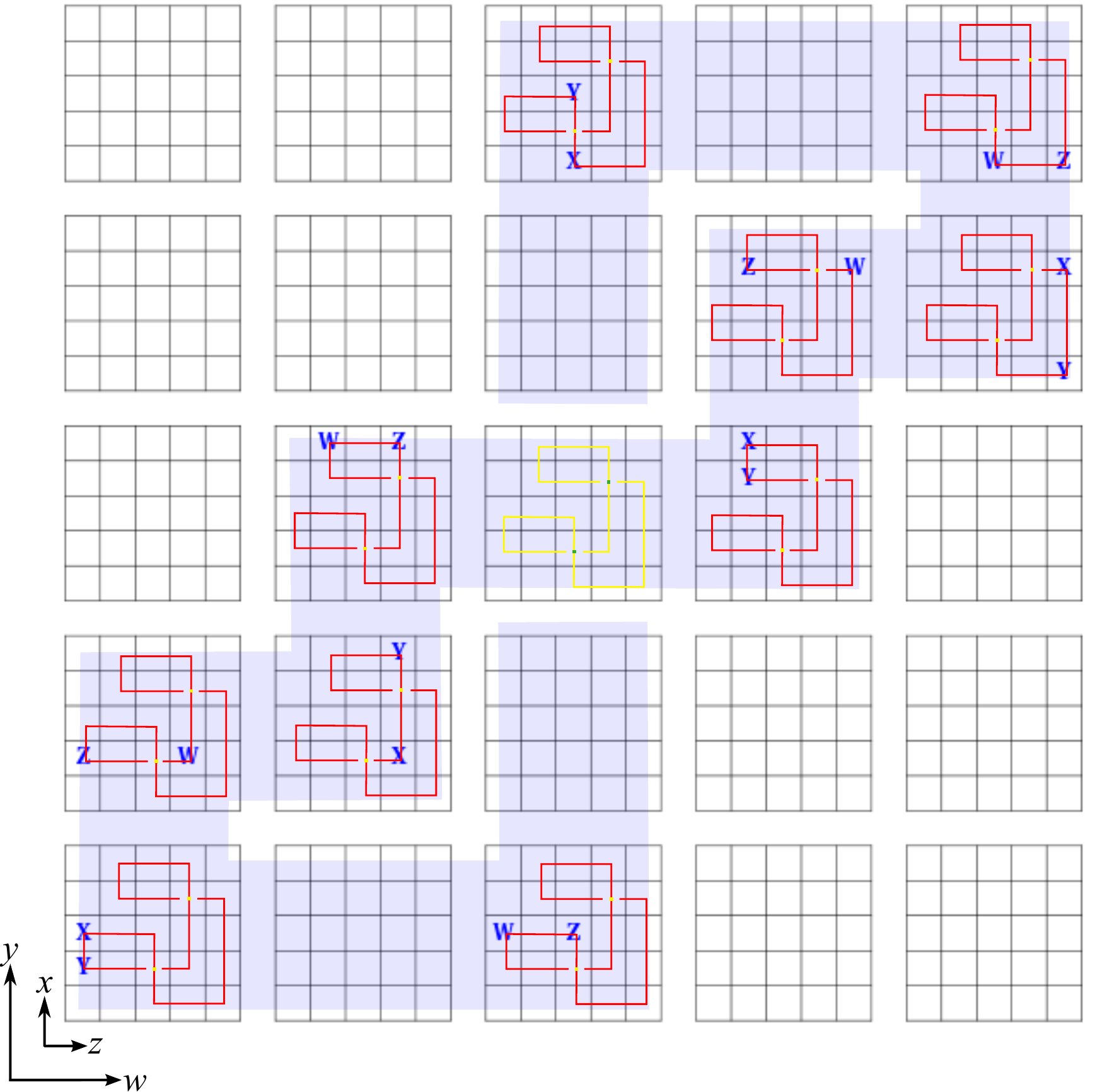}
\caption{Lagrangian hypercube diagram with unknotted $G_{zx}$ and $G_{wy}$ and rotation class $(1,0)$.}
\label{fig:unknot3}
\end{figure}

Furthermore, the torus has only two types of singularities:  double point circles and intersections of double point circles.  Each crossing of $G_{zx}$ generates a double point circle as shown by the yellow dots in Figure \ref{fig:unknot3}.  Similarly each crossing of $G_{wy}$ generates a double point circle, which is visible in the schematic as the $zx$-flat where a $w$-parallel tube passes through a $y$-parallel tube.  In Figure \ref{fig:unknot3} this is shown by the yellow diagram. The green dot in Figure \ref{fig:unknot3} corresponds to an intersection of two double point circles.

\bigskip
\section{Lifting the hypercube to $\mathbb{R}^5$}
\label{section:lift}
\bigskip

Let $i: T \rightarrow \mathbb{R}^4$ be the immersed torus obtained from a Lagrangian hypercube diagram as given by Theorem \ref{thm:torus}.  Note that, $d\alpha |_{wxyz-hyperplane} = \omega = dw \wedge dy + dz \wedge dx$ is a symplectic form on $\mathbb{R}^5$.  We will show that $H\Gamma$ represents the Lagrangian projection of a Legendrian surface in $\mathbb{R}^5$ with respect to the standard contact structure $\xi$.

\medskip

In order to lift $i(T)$ we begin by choosing some point $p \in i(T)$ to have $t$ coordinate equal to some $t_0 \in \RR$.  If we attempt to lift $i(T)$ to a Legendrian surface with respect to $\alpha$ we should choose to define the $t$-coordinate of $p' \neq p$ to be:
\begin{equation}
\label{eqn:lift4D}
t = t_0 + \int_{\gamma}ydw + \int_{\gamma}xdz,
\end{equation}
where $\gamma$ is a path from $p$ to $p'$.  This integral will be independent of path precisely when the $1$-form $i^*(ydw + xdz)$ is $0$ on $H_1(T)$.  Recall that $H_1(T)$ is generated by $\gamma_{zx}$ and $\gamma_{wy}$.

In order check for path-independence of the integral in Equation \ref{eqn:lift4D}, we evaluate the following:
\begin{equation}
i^*(ydw + xdz)[i^*(\gamma_{zx})] = \int_{i^*(\gamma_{zx})}i^*(ydw + xdz) = \int_{\gamma_{zx}}ydw + \int_{\gamma_{zx}}xdz = \int_{\gamma_{zx}}ydw.
\label{IntMeridian}
\end{equation}
\begin{equation}
i^*(ydw + xdz)[i^*(\gamma_{wy})] = \int_{i^*(\gamma_{wy})}i^*(ydw + xdz) = \int_{\gamma_{wy}}ydw + \int_{\gamma_{wy}}xdz = \int_{\gamma_{wy}}xdz.
\label{IntLongitude}
\end{equation}
Since $G_{zx}$ and $G_{wy}$ are Lagrangian grid diagrams, these integrals will both evaluate to $0$ and we get a well-defined lift to a Legendrian torus in $\mathbb{R}^5$ using Equation \ref{eqn:lift4D}.  Furthermore, the Lagrangian crossing conditions guarantee that the lift will be embedded.  Let $L$ be the lift of $i(T)$ obtained from Equation \ref{eqn:lift4D}.  Define $\pi_t : \mathbb{R}^5 \rightarrow \mathbb{R}^4$ to be the projection $(w,x,y,z,t) \mapsto (w,x,y,z)$. Then $\pi_t(L) = i(T)$, i.e. the torus determined by $H\Gamma$ is the Lagrangian projection of the Legendrian torus $L$.  Thus we obtain the following:

\begin{theorem}
\label{thm:lift}
The torus determined by a Lagrangian hypercube diagram $H\Gamma$ lifts to an embedded Legendrian torus $L \subset (\RR^5, \xi)$.  Furthermore, the generators $\gamma_{zx}$ and $\gamma_{wy}$ lift to curves $\tilde{\gamma}_{zx}$ and $\tilde{\gamma}_{wy}$ that generate $H_1(L)$.
\end{theorem}

\begin{remark}
\label{rmk:lagCross}
If we omit the Lagrangian crossing conditions from the definition of a Lagrangian hypercube diagram, then the above procedure will still produce an \emph{immersed} Legendrian torus in $\RR^5$, but it will not, in general, be embedded.
\end{remark}

\begin{example}
\label{ex:3x3unknot}

Figure \ref{fig:unknot3} shows a schematic picture of a Lagrangian hypercube diagram where all grid-projections are unknots as in Example \ref{ex:r0}.  By Lemma \ref{thm:lift}, the torus determined by this Lagrangian hypercube diagram lifts to a Legendrian torus in $(\RR^5, \xi)$.

\end{example}

\section{Proof of Theorem \ref{thm:main1}}
\label{section:compute}

With the rotation class understood to be an element of $[T, U(2)]$ we see from Theorem \ref{isom} that the class may be identified with a pair of integers corresponding to the elements of $\pi_1 (U(2))$ determined by a meridian and longitude of the torus.  Before proving Theorem \ref{thm:main1} we identify an explicit generator of $\pi_1(U(2))$.  Recall that $U(2)$ parametrizes \emph{framed} Lagrangians of $(\RR^2, \omega)$.  Identify the $yx$, $xy$, $yz$, and $zy$ planes with the following matrices:
\[U_{xy} = \left( \begin{array}{cc}
0 & i  \\
i & 0 \end{array} \right),
U_{yx} = \left( \begin{array}{cc}
 0 & i  \\
-i & 0 \end{array} \right), 
U_{yz} = \left( \begin{array}{cc}
 0 & i  \\
-1 & 0 \end{array} \right),
U_{zy} = \left( \begin{array}{cc}
0 & i  \\
1 & 0 \end{array} \right). \]
Note that $U_{xy}$, $U_{xy}$, $U_{yz}$, and $U_{zy}$ correspond to unitary Lagrangian frames (cf. \cite{mcduff}):
\[U_{xy} \mapsto \left( \begin{array}{cc}
0 & 0  \\
0 & 0  \\
0 & 1  \\
1 & 0  \end{array} \right),
U_{yx} \mapsto \left( \begin{array}{cc}
 0 & 0  \\
 0 & 0  \\
 0 & 1  \\
-1 & 0  \end{array} \right), 
U_{yz} \mapsto \left( \begin{array}{cc}
 0 & 0  \\
-1 & 0  \\
 0 & 1  \\
 0 & 0  \end{array} \right),
U_{zy} \mapsto \left( \begin{array}{cc}
 0 & 0  \\
 1 & 0  \\
 0 & 1  \\
 0 & 0  \end{array} \right),\]
 Note that as maps from $\RR^2 \rightarrow \RR^4$ these frames produce $xy$, $(-x)y$, $(-z)y$, and $zy$-planes respectively.  Geometrically, this matches up with the fact that the Lagrangian planes along an $xz$-slice of the hypercube will be given by a positively or negatively oriented $\partial_x$ or $\partial_z$ vector paired with a positively oriented $\partial_y$-vector.  

\medskip

Choose $U_{xy}$ to be the basepoint.  We define a loop $\gamma: [0,1] \rightarrow U(2)$ that begins at $U_{xy}$ and rotates through $U_{yz}$, $U_{yx}$ and $U_{zy}$.  We will define $\gamma$ in $4$ pieces.  First, define a map $\hat{\gamma}: [0,1] \rightarrow U(2)$ as follows:
\[\hat{\gamma} (t) = \left( \begin{array}{cc}
1 & 0  										\\
0 & e^{\frac{\pi}{2} i t} \end{array} \right).\]
Then, define $\gamma_1(t) = \hat{\gamma}(t)U_{xy}$, $\gamma_2(t) = \hat{\gamma}(t)U_{yz}$, $\gamma_3(t) = \hat{\gamma}(t)U_{yx}$, and $\gamma_4(t) = \hat{\gamma}(t)U_{zy}$.  Finally, define $\gamma(t) = \gamma_1 \star \gamma_2 \star \gamma_3 \star \gamma_4$. Thus $\gamma$ corresponds to a rotation of Lagrangian planes, beginning at an $xy$-plane, and rotating through $yz$, $yx$, and $zy$-planes.

\begin{lemma}
\label{lemma:gen}
The loop $\gamma$ represents a generator of $\pi_1(U(2))$.
\end{lemma}

\begin{proof}
Observe that the determinant, $det:  U(2) \rightarrow U(1)$  induces an isomorphism on $\pi_1$ that takes $\gamma$ to a generator of $\pi_1(U(1))$.
\end{proof}

The same argument will show that there is a generator for $\pi_1(U(2))$ given by acting on matrices $U_{xy}$, $\hat{U}_{yx}$, $U_{xw}$, and $U_{wx}$ on the left by:
\[\tilde{\gamma} (t) = \left( \begin{array}{cc}
e^{\frac{\pi}{2} i t} & 0  										\\
0 										& 1 \end{array} \right).\]
Note that $U_{yx} \neq \hat{U}_{yx}$ as matrices in $U(2)$ but they give rise to the same Lagrangian planes, with the same orientation.  While $U_{yx}$ corresponds to a unitary Lagrangian frame giving rise to the Lagrangian plane $\{ -\partial_x, \partial_y\}$, $\hat{U}_{yx}$ gives rise to the Lagrangian plane $\{\partial_x, -\partial_y\}$.

\medskip

Much of the content of the paper to this point has been building up toward presenting the following proof. Our discussion of Lagrangian grid diagrams in Section \ref{section:grids} enables us to define an immersed Lagrangian torus corresponding to a Lagrangian hypercube diagram as in Theorem \ref{thm:torus}.  Lemma \ref{thm:lift} shows how to obtain a Legendrian torus from the Lagrangian hypercube diagram.  Having determined easy methods for computing the rotation number of the Lagrangian grid diagrams (Lemma \ref{lemma:rotation}), we are ready to prove Theorem \ref{thm:main1}.
 
\medskip

\begin{proof}
Lemma \ref{thm:lift} guarantees that the lift, $L$, exists.  We must see that the image of $r(L) \in [T, U(2)]$ under the isomorphism defined in Theorem \ref{isom} is $(w(G_{zx}), w(G_{wy}))$.  $G_{zx}$ and $G_{wy}$ each correspond to one of the two factors of $T$.  Let $[f_{zx}]$ and $[f_{wy}]$ be the elements of $\pi_1(U(2))$ determined by $G_{zx}$ and $G_{wy}$ (since $G_{zx}$ and $G_{wy}$ are constant, choice of base point is irrelevant).  Then the isomorphism defined in Theorem \ref{isom} maps $r(L)$ to $([f_{zx}],[f_{wy}])$.  We must show that $[f_{zx}] = w(G_{zx}) [\gamma]$.

Clearly, $w(G_{zx})$ computes how many times the tangent vector to the grid $G_{zx}$ wraps around the loop $\gamma$.  By Lemma \ref{lemma:gen} $[\gamma]$ generates $\pi_1(U(2))$.  A similar argument shows that $[f_{wy}] = w(G_{wy}) [\gamma]$.
\end{proof}

\medskip

\noindent {\bf Corollary \ref{cor:main11}} {\it Let $H_1(T_{H\Gamma})$ be generated by $i(\gamma_1)$ and $i(\gamma_2)$ (as in Theorem \ref{thm:torus}).  The Maslov index, $\mu: H_1(T_{H\Gamma}) \rightarrow \ZZ$ can be computed directly.  For $A = (a,b) \in H_1(T_{H\Gamma})$, 
$$\mu(A) = 2aw(G_{zx}) + 2bw(G_{wy})$$.}

\begin{proof}
Given an embedded loop $\gamma : S^1 \rightarrow T_{H\Gamma}$ representing a primitive class $A \in H_1(T_{H\Gamma})$, for any $p \in S^1$, $T_{\gamma(p)}T_{H\Gamma}$ is a Lagrangian plane, $L_{\gamma(p)}$.  Thus we obtain a map $S^1 \rightarrow Lag(\CC^2)$ such that $p  \mapsto L_{\gamma(p)}$.  The isomorphism defined in the proof of Theorem \ref{thm:main1} is valid here as well, once we identify planes that differ only in orientation, which produces a factor of $2$.
\end{proof}

\noindent {\bf Corollary \ref{cor:main12}} {\it The Maslov number is $2gcd(w(G_{zx}), w(G_{wy}))$}.

\begin{proof}
Follows directly from the previous corollary and the fact that the Maslov number is the smallest positive number that is the Maslov index of a non-trivial loop in $H_1(T_{H\Gamma})$ and $0$ if every non-trivial loop has Maslov index $0$ (cf. \cite{rotation}.  
\end{proof}

\bigskip
\section{Proof of Theorem \ref{thm:main2} and Examples}
\label{section:examples}
\bigskip

Before proceeding with the proof of Theorem \ref{thm:main2} we establish a few preliminary results.  The construction of Theorem \ref{thm:construction} can be used to produce a hypercube diagram (in the sense of \cite{Bald}) given any pair of Lagrangian grid diagrams.  However if the Lagrangian crossing conditions are not satisfied by the pair of Lagrangian grid diagrams, the resulting Legendrian torus will not be embedded (cf. Remark \ref{rmk:lagCross}).  Theorem \ref{thm:expandCrossings}, \ref{thm:stabilize}, and Corollary \ref{cor:stabilize} show that for any pair of topological knots, and any rotation numbers, one may find a pair of Lagrangian grid diagrams such that the Lagrangian crossing conditions are satisfied and hence construct a Lagrangian hypercube diagram that lifts to an embedded Legendrian torus.  

\begin{theorem}
\label{thm:expandCrossings}
Let $G$ be a Lagrangian grid diagram with an upper-right corner.  Enumerate the crossings of $G$ by $\{c_i\}$.  Then, for any $M > 0$ there is another Lagrangian grid diagram $G'$, representing the same topological knot and having the same rotation number as $G$, such that $|\Delta t(c_i')| > M$ for all $i$.  
\end{theorem}

\begin{proof}
Scale $G$ by $k \in \ZZ$ (each segment of the diagram of length $\ell$ becomes a segment of length $k \ell$).  This produces a diagram satisfying the Lagrangian conditions (Equations \ref{eqn:areaInt} and \ref{eqn:crossingInt}), but, of course, it will not be a grid diagram, due to empty rows and columns.  However, the area of each rectangle (as in Corollary \ref{cor:areaCalc}) will be multiplied by $k^2$.  Therefore, $|\Delta t(c_i)|$ may be made arbitrarily large for all $i$.  We must then show that the empty rows and columns may be filled in, while preserving the Lagrangian grid conditions.

By following the techniques of Theorem \ref{thm:anyRotation} we may assume that the upper-right corner of $G$ (prior to scaling) has a horizontal and vertical edge of length $1$ or $2$.  Begin by inserting one additional row and column at the upper-right corner.  The additional area created by this will be either $2k+1$, $3k+1$, or $4k+1$ depending on the initial lengths of the horizontal and vertical edges of the upper-right corner.  Then attach the configuration shown in Figure \ref{fig:gullWings}.  The unshaded regions will be equal in area, but with opposite sign due to the symmetry between empty rows and columns after scaling the initial grid. The dark-grey regions will also be equal in magnitude but with opposite sign.  Finally the light-grey region at the top right may be extended so that it is of area $2k+1$, $3k+1$, or $4k+1$ (an even or odd area may be acheived by placing an additional box as shown bythe dotted lines at the upper-right corner of Figure \ref{fig:gullWings}).

Finally, observe that for all of the original crossings, $\Delta t$ has been scaled up by a factor of $k^2$.  However, this procedure creates $4$ additional crossings:  $d_1$, $d_2$, $d_3$, and $d_4$.  By choosing $k$ sufficiently large, and possibly making our initial grid diagram larger, we may ensure that $min{|\Delta t(d_i)|} \geq jk+1$ for $j = 2,3,4$.
\end{proof}

\begin{figure}[h]
\includegraphics[scale = 1]{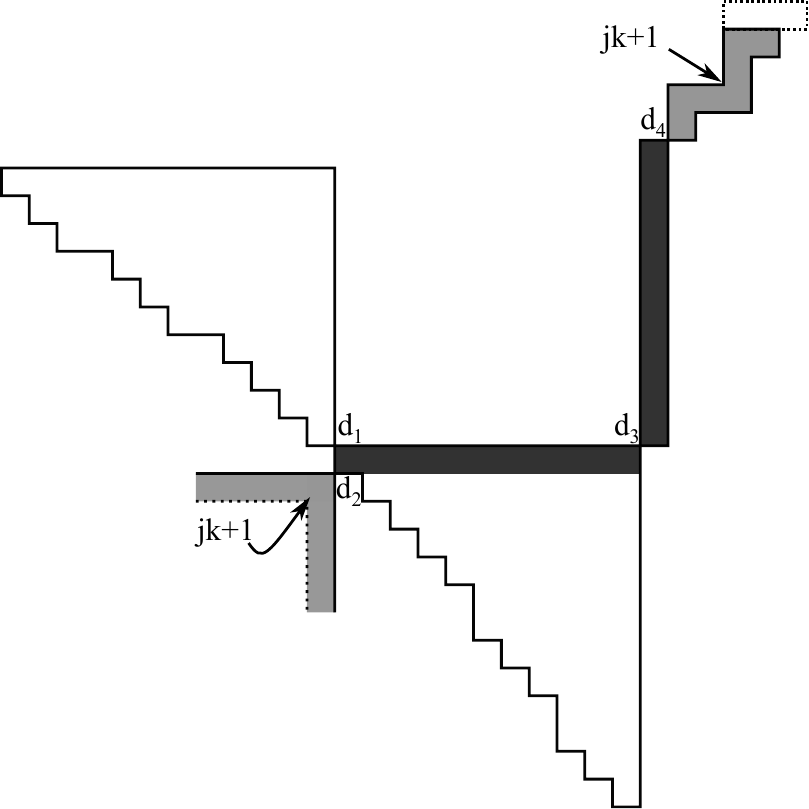}
\caption{Configurations used to fill in empty rows and columns ($j = 2,3,4$).}
\label{fig:gullWings}
\end{figure}

We showed in the previous theorem that the minimum value of $|\Delta t(C_i)|$ may be made arbitrarily large for a Lagrangian grid diagram, the following theorem shows that we may make Lagrangian grid diagrams arbitrarily large, while keeping $\Delta t(c_i)$ small.  

\begin{theorem}
\label{thm:stabilize}
Given a Lagrangian grid diagram $G$ of size $n$, there exists $m > n$ such that one may modify $G$ to obtain a Lagrangian grid diagram, $G'$ of size $n'$ for any $n' > m$, with the same topological type and rotation number as $G$.  Moreover, if $\Delta_1$ is the maximum over $|\Delta t (c_i)|$ for $G$ and $\Delta_2$ is defined similarly for $G'$, then $\Delta_2 \leq \Delta_1 + |a| + 1$.
\end{theorem}

\begin{figure}[h]
\includegraphics[scale = .7]{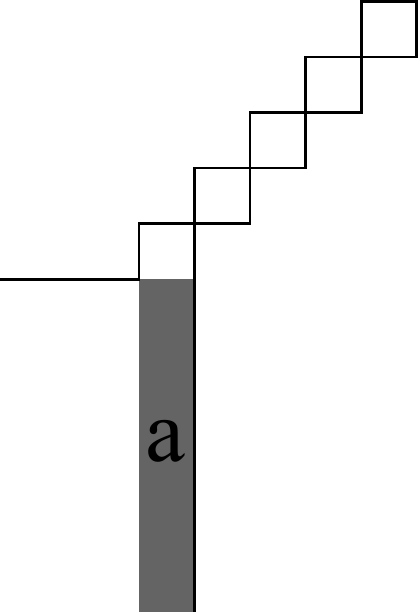}
\caption{Configuration used to enlarge a Lagrangian grid diagram.}
\label{fig:stabilize}
\end{figure}

\begin{proof}
We may assume that $G$ has an upper-right corner.  Let $k \in \ZZ$.  At the top right corner of the grid, we stabilize and attach a configuration of size $2k$ as shown in Figure \ref{fig:stabilize}.  Since we began with a Lagrangian grid diagram, each new crossing created in this procedure will have $|\Delta t|$ equal to either $a \pm 1$ or $a$, and at the new top right corner, the $t$-coordinates will differ by $a \pm 1$.  We then apply Lemma \ref{lemma:capArc} to obtain a Lagrangian grid diagram. By carefully choosing wwhich configurations we use in applying Lemma \ref{lemma:capArc}, we may ensure that the Lagrangian grid diagram we obtain has even or odd size.  The statement about the bound on $\Delta_2$ is clear from the construction.
\end{proof} 

\begin{corollary}
\label{cor:stabilize}
Given two Lagrangian grid diagrams, $G$ and $G'$ of size $m$ and $n$, they may be stabilized to obtain Lagrangian grid diagrams representing the same two topological knots, without changing the rotation number, and such that if ${c_i}$ is the set of crossings in $G$ and ${c_j'}$ is the set of crossings in $G'$, $|\Delta t(c_i)| < |\Delta t(c_j')|$ for all $i,j$.
\end{corollary}

\begin{proof}
Apply Theorem \ref{thm:expandCrossings} to $G$, choosing $k$ sufficiently large to guarantee that $k^2 > 4k+1$ and $2k+1 > max\{\Delta t(c_i')\} + |a| + 1$ where $a$ is as shown in Figure \ref{fig:stabilize}.   This guarantees that $min\{\Delta t(c_i)\} > max\{\Delta t(c_i')\} + |a| + 1$.  Then apply Theorem \ref{thm:stabilize} to $G'$ so that both grids are the same size. 
\end{proof}

\begin{theorem}
\label{thm:construction}
Let $G_{wy}$ and $G_{zx}$ be Lagrangian grid diagrams of the same size such that if ${c_i}$ is the set of crossings in $G_{wy}$ and ${c_j'}$ is the set of crossings in $G_{zx}$, then $|\Delta t(c_i)| \neq |\Delta t(c_j')|$ for all $i,j$.  Then, there is a Lagrangian hypercube diagram such that the $wy$ an $zx$-projections are given by these grids.
\end{theorem}

\begin{proof}
Following the orientation of the diagram label the markings $W_0, Y_0, W_1, Y_1,...$ etc.  Do the same for $G_{zx}$.  Denote the coordinates of $W_i$ by $(w_{w,i},y_{w,i})$, $Y_i$ by $(w_{y,i},y_{y,i})$ etc.  Place $Z_i$ in the hypercube at position $(w_{w,i},x_{z,i},y_{w,i},z_{z,i})$, $W_i$ at position $(w_{w,i},x_{x,i},y_{w,i},z_{x,i})$, $X_i$ at position $(w_{y,i},x_{x,i},y_{y,i},z_{x,i})$,and $Y_i$ at position $(w_{y,i},x_{z,i+1},y_{y,i},z_{z,i+1})$ where $i$ is taken modulo $n$.
\end{proof}

Having developed the results on Lagrangian grid diagrams in Section \ref{section:grids}, and having shown in Theorems \ref{thm:construction}, \ref{thm:stabilize}, and Corollary \ref{cor:stabilize} we now have the necessary framework to complete the proof of Theorem \ref{thm:main2} below.

\begin{proof}
Given $(m,k) \in \ZZ^2$, and two knot types $K_1$ and $K_2$.  Theorem \ref{thm:anyRotation} allows one to construct Lagrangian grid diagrams $G_1$ and $G_2$ representing $K_1$ and $K_2$ with rotation numbers $m$ and $k$ respectively.  Corollary \ref{cor:stabilize} allows one to find Lagrangian grid diagrams, $G_1'$ and $G_2'$, of the same size representing the same topological knots and having the same rotation numbers as $G_1$ and $G_2$.  Applying Theorem \ref{thm:construction} enables us to construct a Lagrangian hypercube diagram such that $G_{zx} = G_1'$ and $G_{wy} = G_2'$.
\end{proof}

\begin{example}
\label{ex:0N}
One may construct a Lagrangian grid diagram for the unknot with arbitrary rotation number by following the construction shown in Figure \ref{fig:rotationN}.  To realize rotation number $r > 0$ construct the diagram as in Figure \ref{fig:rotationN} using $r+1$ horizontal bars of length $r$.  The resulting diagram will have size $2r+3$.  Let $G_{zx}$ be such a grid diagram.  Let $G_{wy}$ be the Lagrangian grid diagram for the unknot of size $2r+3$ given by the construction shown in Figure \ref{fig:rotation0}.  Then applying Theorem \ref{thm:construction}, Lemma \ref{thm:lift} and Theorem \ref{thm:main1} we obtain a Lagrangian hypercube diagram with rotation class $(r,0)$.  Figure \ref{fig:unknot3} shows the construction for $r = 1$.

Note that if $r = 0$ one must first apply Corollary \ref{cor:stabilize}.  However, for $r > 1$, $|\Delta t(c_i)|$ is never equal to $|\Delta t(c)|$ where $c$ is the unique crossing in $G_{zx}$ and $\{c_i\}$ is the set of crossings in $G_{wy}$.  

\begin{figure}[h]
\includegraphics[scale=.5]{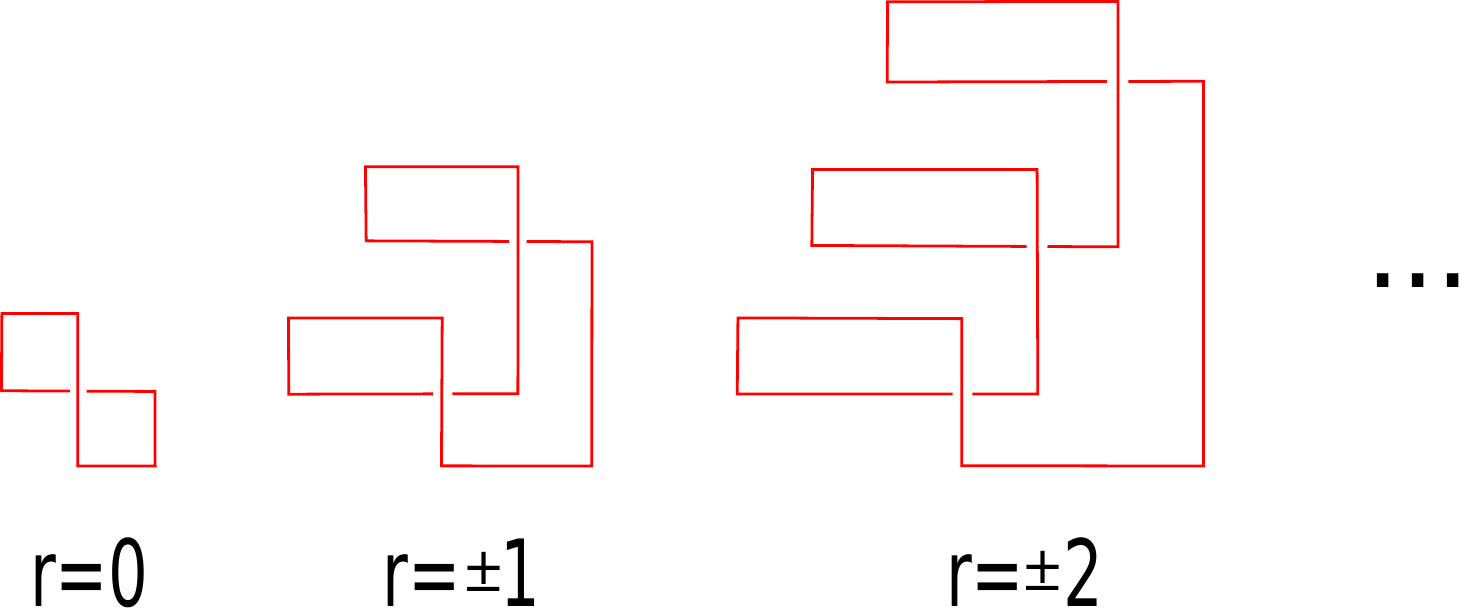}
\caption{Construction of a Lagrangian unknots with rotation number $0, \pm 1, \pm 2$.}
\label{fig:rotationN}
\end{figure}

\begin{figure}[h]
\includegraphics[scale=.5]{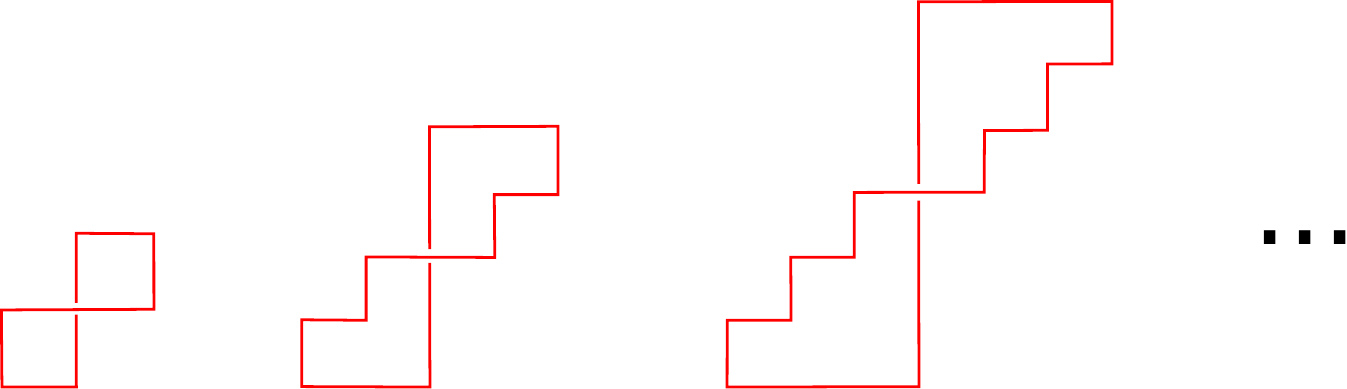}
\caption{Construction of a Lagrangian unknots with rotation number $0$.}
\label{fig:rotation0}
\end{figure}

\end{example}

\begin{example}
 
Figure \ref{fig:nontrivialhyper} shows a Lagrangian hypercube diagram with $G_{zx}$ representing a trefoil, and $G_{wy}$ representing a $(5,2)$ torus knot.  One may check that $G_{wy}$ has rotation number $0$, $G_{zx}$ has rotation number $1$, and hence, the Lagrangian hypercube diagram has rotation class $(1,0)$.  

\begin{figure}[h]
\includegraphics[scale=.4]{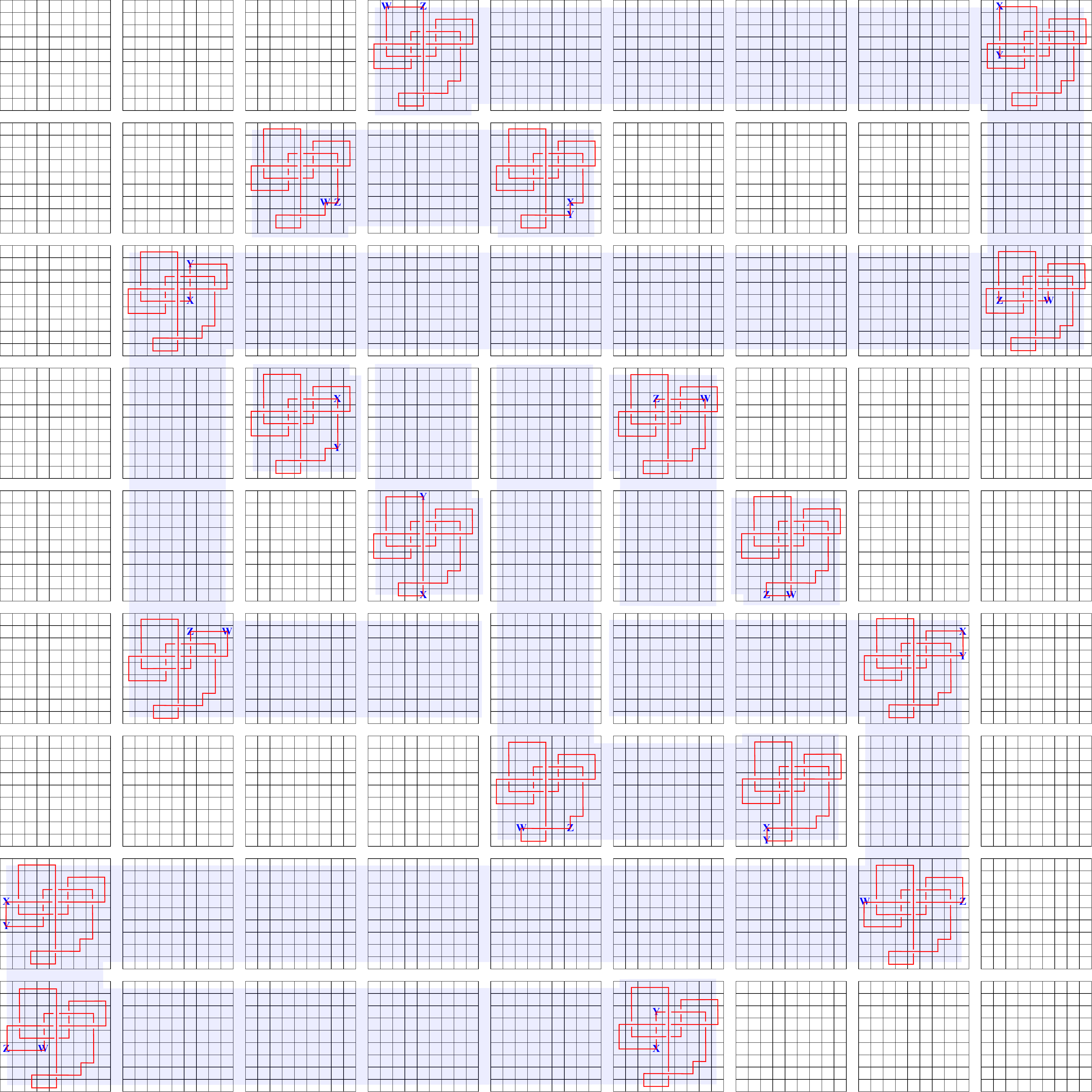}
\caption{Hypercube diagram with $G_{wy}$ representing a (5,2) torus knot, and $G_{zx}$ representing a trefoil.}
\label{fig:nontrivialhyper}
\end{figure}

\end{example}

\end{document}